\documentclass[12pt]{amsart}
\usepackage{amsmath,amssymb,latexsym,amsmath,amsthm,amscd}
\usepackage{setspace}
\usepackage[all]{xy} \xyoption{arc}
\usepackage[left=3cm,top=2cm,right=3cm,bottom = 2cm]{geometry}
\usepackage{graphicx}
\usepackage[usenames,dvipsnames]{color}
\usepackage{hyperref}
\usepackage{charter}

\theoremstyle{plain}
\newtheorem{thm}{Theorem}
\newtheorem{lem}[thm]{Lemma}
\newtheorem{prop}[thm]{Proposition}

\theoremstyle{definition}

\theoremstyle{remark}
\newtheorem{rmk}[thm]{Remark}

\newcommand{\CC}{\mathbb{C}}

\newcommand{\ZZ}{\mathbb{Z}}

\newcommand{\RR}{\mathbb{R}}

\newcommand{\vac}{\mathbf{1}}
\newcommand{\Nfour}{\mathcal{A}_{N=4}}
\newcommand{\Ntwo}{\mathcal{A}_{N=2}}
\newcommand{\half}{\frac{1}{2}}
\newcommand{\abar}{\overline{a}}
\newcommand{\bbar}{\overline{b}}
\newcommand{\dlangle}{\langle\hspace{-0.8 mm}\langle} 
\newcommand{\drangle}{\rangle\hspace{-0.8 mm}\rangle}

\DeclareMathOperator{\End}{End}
\DeclareMathOperator{\Id}{Id}
\DeclareMathOperator{\slt}{\mathfrak{sl}_2}
\DeclareMathOperator{\slthat}{\widehat{\mathfrak{sl}}_2}

\begin{document}
\title{$N{=}2$ and $N{=}4$ subalgebras of super vertex operator algebras}
\author{Geoffrey Mason, Michael Tuite and Gaywalee Yamskulna}
\address{Department of Mathematics\\
University of California, Santa Cruz}
\email{gem@ucsc.edu}
\address{School of Mathematics, Statistics and Applied Mathematics\\ 
National University of Ireland Galway, Galway, Ireland}
\email{michael.tuite@nuigalway.ie}
\address{Department of Mathematics\\ Illinois State University, Normal, IL}
\email{gyamsku@ilstu.edu}
\thanks{The first author was supported by the NSF and the Simons Foundation \#427007.
The third author was supported by the Simons Foundation Collaboration Grant \# 207862.}

\subjclass{Primary, Secondary }

\begin{abstract} We develop criteria to decide if an $N{=}2$ or $N{=}4$ super conformal algebra
is a subalgebra of a super vertex operator algebra in general, and of a super lattice theory in particular.\ We give some specific examples.
\end{abstract}
\maketitle
 
\section{Introduction}
The advent of Mathieu Moonshine \cite{EOT} in recent years has brought renewed attention to the
super conformal $N{=}4$ algebra $\Nfour$.
Song has shown \cite{S} that for central charge $c{=}6$, $\Nfour$ is the algebra of global sections of the chiral de Rham complex on a $K3$ surface (following earlier results for hyperk\"ahler manifolds \cite{BZHS}), clarifying the
connection between $K3$ and $\Nfour$ as long understood by physicists. 

\medskip
All of this means that one can expect $\Nfour$ to play a ubiquitous r\^{o}le in the further investigation of
these subjects, much as the Virasoro algebra does in general CFT.\ Now the precise definition
of $\Nfour$ is awkward, to say the least. For  $c{=}6$,  it is usually described as a subtheory of the algebra of
$12$ free fermions.\ When we found ourselves looking for $\Nfour$ in a super lattice theory 
containing no free fermions \cite{MTY}, there was a dearth of results in the literature to which we could turn.\

\medskip
The purpose of the present paper is to alleviate this situation.\ We prove two general recognition theorems which allow
one to identify $\Nfour$ as a subalgebra of a suitable Super Vertex Operator Algebra (SVOA) with just a few well-chosen axioms.\ Experts
will perhaps not be surprised by the results, however some effort is required to obtain efficient characterizations, i.e., without too many assumptions.\  It would be of interest if a genuine reduction in the number of
axioms needed can be achieved in our  recognition theorems.\ Our main results are as follows (unexplained notation is clarified below).
\begin{thm}\label{thmUN=4} 
Let $U$ be a SVOA of CFT--type.\ 
Let  $V{\subseteq} U$ be the subalgebra generated by 4 \emph{primary vectors} of weight $\frac{3}{2} $ in $U$, so that 
\begin{align*}
V{=}\CC\mathbf{1} {\oplus} V_{\frac{1}{2}} {\oplus} V_1{\oplus} V_{\frac{3}{2} } {\oplus}\hdots
\end{align*}
is a conformally graded subspace of $U$.\ Assume that the following  hold:
\begin{enumerate}
	\item[(I)] The subspace of $V_{\frac{3}{2} }$ spanned by the four generators decomposes as $A{\oplus}B$, a pair of 2-dimensional vector representations for $\slt$, 
	where
	\item[(II)] $A(1)B\cong \slt$, 
	\item[(III)]  $A(0)A{=}B(0)B{=}0$,
  \item[(IV)] $ T\slt \cap A(0)B\neq 0$. 
\end{enumerate}
Then $V{\cong} \Nfour$ with central charge $c{=}6k$, where $k{\in}\CC$ is the level of the $\slthat$ Kac-Moody subalgebra generated by $A(1)B$. 
\end{thm}

\begin{thm}\label{thmVLV} Let $L$ be a positive-definite, odd, integral lattice of minimum norm $2$ with $V_L$ the corresponding SVOA.\ Let  $V{\subseteq} V_L$ be the subalgebra generated by 4 vectors of weight $\frac{3}{2} $, so that
\begin{align*}
V{=}\CC\mathbf{1} {\oplus} V_1{\oplus} V_{\frac{3}{2} } {\oplus}\hdots,
\end{align*}
and assume that the following  hold:
\begin{enumerate}
	\item[(I)] The subspace of $V_{\frac{3}{2} }$ spanned by the four generators decomposes as $A{\oplus}B$, a pair of 2-dimensional vector representations for $\slt$, 
	where
	\item[(II)] $A(1)B\cong \slt$,
	\item[(III)] $A(1)A{=}B(1)B{=}0$,
	\item[(IV)] $\slt$ contains a root of $L$.
\end{enumerate}
Then $V\cong \Nfour$ with central charge $c{=}6$.
\end{thm}
The paper is organized as follows.\ After reviewing some background about the $N{=}4$ algebra in Section \ref{SN=4}, we give the proof of Theorems \ref{thmUN=4} and \ref{thmVLV} in Sections \ref{SRecog1} and \ref{SRecog2} respectively.\
We also briefly describe the more elementary analogous case of the $N{=}2$ superconformal subalgebra.\ This is contained in Section \ref{SN=2}.\
 In Section \ref{SExample} we illustrate how the main Theorems may be applied to some examples.\ In particular,
we give (Proposition \ref{thmalt}) a painless new construction of the $N{=}4$ algebra in a certain rank 6 lattice theory $V_L$ containing no free fermions.\ We include some appendices in Section \ref{SAppendix} containing technical background in SVOA theory
that we assume and use throughout the paper.

\noindent \textbf{Acknowledgements} We thank J.~Duncan, R.~Heluani, M.~Miyamoto and R.~Volpato for their comments and corrections to an earlier version of this paper.
\medskip

\section{$N{=}4$ superconformal algebras}\label{SN=4}
The abstract generators and relations for the $N{=}4$ super conformal algebra $\Nfour$  of central charge $c$ are as follows.\ It is generated by 4 states $G^{\pm}, \overline{G}^{\pm}$ of conformal weight $\frac{3}{2} $.\ The nontrivial relations can be expressed as follows (e.g. \cite{ET,K})
\begin{align*}
&\mbox{(a)}\; && J^{0}(0)J^{\pm}{=}{\pm}2J^{\pm},\quad &
&\mbox{(b)}\; && J^{0}(1)J^{0}{=}\frac{c}{3}\vac, \quad &
&\mbox{(c)}\; && J^{+}(0)J^{-}{=}J^{0},
 \\
&\mbox{(d)}\; && J^{+}(1)J^{-}{=}\frac{c}{6}\vac,\quad &
&\mbox{(e)}\; && J^{0}(0)G^{\pm}{=}{\pm} G^{\pm}\quad & 
&\mbox{(f)}\; && J^{0}(0)\overline{G}^{\pm}={\pm}\overline{G}^{\pm}, 
\\ 
&\mbox{(g)}\; && J^{\pm}(0)G^{\mp}{=}G^{\pm}, \quad &
&\mbox{(h)}\; && J^{\pm }(0)\overline{G}^{\mp}{=}{-}\overline{G}^{\pm}, \quad &
&\mbox{(i)}\; && G^{\pm}(1)\overline{G}^{\pm}{=}2J^{\pm},
\\
&\mbox{(j)}\; && G^{\pm}(1)\overline{G}^{\mp}{=}{\pm}J^{0},\quad &
&\mbox{(k)}\; && G^{\pm}(2)\overline{G}^{\mp}{=}\frac{c}{3}\vac,\quad &
&\mbox{(l)}\; && G^{\pm}(0)\overline{G}^{\pm}{=}TJ^{\pm},
\\
&\mbox{(m)}\quad && G^{\pm}(0)\overline{G}^{\mp}{=}\omega{\pm}\frac{1}{2}TJ^{0}.\notag
\end{align*}
Here, $\omega $ is a Virasoro element of central charge $c$, $T$ is the translation operator and  $J^{\pm},J^{0}$ are weight one vectors.\ Relations (a)--(m) hold in addition to the usual Virasoro relations between $J^{\pm},J^{0}$ and $\omega$.\ The initial segment of the Fock space of $V{=}\Nfour$ is
\begin{align*}
V{=}\CC\mathbf{1}\oplus V_1\oplus V_{\frac{3}{2} }\oplus V_2\oplus\hdots
\end{align*}
where
\begin{align*}
&V_1{=}\langle J^{\pm}, J^0\rangle {\cong} \slt\\
&V_{\frac{3}{2} }{=}\langle G^{\pm}\rangle{\oplus}\langle\overline{G}^{\pm}\rangle,\  \mbox{a pair of vector representations for 
$\slt$}\\
&V_2{=} T\slt{\oplus} \CC\omega .
\end{align*}

\medskip
For arbitrary subsets $X, Y{\subseteq}U$ and any integer $n$, we have already used, and will continue to use, the convenient notation $X(n)Y$
for the linear subspace of $U$ spanned by all products $u(n)v\ (u{\in}X, v{\in}Y)$.\
For the $N{=}4$ algebra $\Nfour$ we set 
\begin{align*}
A{:=}\langle G^{\pm}\rangle,\ B{:=}\langle \overline{G}^{\pm}\rangle.
\end{align*}
Then it follows from the relations (a)--(m) that
\begin{align*}
A(n)A{=}B(n)B{=}0\ (n\geq 0),\quad A(0)B{=}V_2,\quad A(1)B{=}\slt,\quad A(2)B{=}\CC c\mathbf{1}.
\end{align*}

\section{Proof of Theorem \ref{thmUN=4}}\label{SRecog1}
Let $\omega^{U}$ be the Virasoro element of $U$, with vertex operator
\begin{align*}
Y(\omega^{U}, z)=\sum_{n\in\ZZ} L^{U}(n)z^{-n-2}.
\end{align*}

\subsection{Assumptions (I) and (II)}
In the interests of keeping track of just which of the axioms (I)--(IV) in the statement of
Theorem \ref{thmUN=4} are needed when, we begin by recording  some consequences of axioms (I) and (II) alone.\ 
We set
\begin{align}
&\slt{:=}\langle h, x^{\pm}\rangle\subseteq V_1 \label{slnames}\\
&A{:=}\langle \tau^+, \tau^-\rangle,\quad B{:=}\langle \overline{\tau}^+, \overline{\tau}^-\rangle. \label{ABnames}
\end{align}
Here, 
\begin{align}h(0)x^{\pm}{=}\pm2 x^{\pm},~ x^+(0)x^-{=}h\label{sl2relation}
\end{align} are standard generators and relations for $\slt$,
and $\tau^{\pm}, \overline{\tau}^{\pm}$ are weight vectors for $h(0)$ with weights $\pm 1$, with
\begin{align}
&x^{\pm}(0)\tau^{\mp}{=}\tau^{\pm},\quad 
x^{\pm}(0)\overline{\tau}^{\mp}{=}\overline{\tau}^{\pm},
\label{xpmtau}
\\
&x^{\pm}(0)\tau^{\pm}=
x^{\pm}(0)\overline{\tau}^{\pm}=0.
\label{xpmtau0}
\end{align}

All of this is just a choice of notation based on the hypotheses of the Theorem \ref{thmUN=4}.\ It amounts to the existence of an
isomorphism of $\slt$-modules
\begin{align*}
\varphi{:}A{\rightarrow}B, \ \ \tau^{\pm}{\mapsto} \overline{\tau}^{\pm}.
\end{align*}
Note that some latitude in scaling the generators of $A$ and $B$ remains - a fact that we make use of later.\
In any case, there is a canonical $\slt$-invariant decomposition into   trivial and adjoint modules
\begin{align*}
A{\otimes}B{=}\Lambda {\oplus}\Sigma,
\end{align*}
where
\begin{align}
&\Lambda{=}\CC(\tau^+{\otimes}\overline{\tau}^-{-}\tau^{-}{\otimes}\overline{\tau}^+) \label{Lambdadef},\\
&\Sigma{=}\langle \tau^{+}{\otimes}\overline{\tau}^{+}, \tau^{-}{\otimes}\overline{\tau}^{-}, 
\tau^+{\otimes}\overline{\tau}^-{+}\tau^-\otimes\overline{\tau}^+\rangle. \label{Sigmadef}
\end{align}
Of course, there is a similar decomposition of $A{\otimes}A{=}\Lambda^2(A){\oplus}S^2(A)$.
\medskip

Note that $A(1)B$ and $A{\oplus}B$ consist of primary states.\ This is clear for $A{\oplus}B$, while
for $A(1)B$ we note that $L^U(1){:}V_1{\rightarrow}\CC\mathbf{1}$ is a morphism of $\slt$-modules.\
Restriction to $A(1)B{=}\slt$ must therefore be trivial, and the assertion follows.\ Thus we have
\begin{align}
[L^{U}(m),u(n)] &{=} {-}nu(m{+}n),
\label{Lun}
\\
[L^{U}(m),v(n)] &{=} \left(\frac{1}{2}(m{+}1){-}n\right)v(m{+}n),
\label{eq:Lvn}
\end{align}
for all for $u\in A(1)B$ and  $v{\in}A{\oplus}B$.
 In particular, \eqref{eq:Lvn} implies
\begin{lem}\label{lem:AnpA}
$A(n{+}1)A{=}L^{U}(1)A(n)A$ for all $n\neq 1$ and $A(n{+}2)A{=}L^{U}(2)A(n)A$ for all $n$. Similar relations hold for $B(n{+}1)B$, $B(n{+}2)B$, $A(n{+}1)B$ and $A(n{+}2)B$.
$\hfill \Box$
\end{lem}
\medskip

For $u,v\in A(1)B{=}\slt$ we have from commutativity that 
\begin{align}
\label{KM}
[u(m),v(n)]{=}(u(0)v)(m+n)+m\dlangle    u,v \drangle  \delta_{m+n,0},
\end{align}
with $\dlangle  u,v \drangle $ given by $u(1)v{=}\dlangle  u,v \drangle \vac$. Define $k{\in}\CC$ by 
\begin{align}
\dlangle  h,h \drangle {=}2k.
\label{eq:kdef}
\end{align}

\begin{lem}\label{lemKM}
$V$ is a module for the Kac-Moody  algebra $ \slthat$ of level $k$.  
\end{lem} 
\begin{proof} 
Since $h(1)x^{\pm}\in V_0=\CC \vac$ we find
\begin{align*}
0=h(0)(h(1)x^{\pm})=h(1)h(0)x^{\pm}=\pm 2 h(1)x^{\pm},
\end{align*}
so that $ \dlangle  h,x^{\pm} \drangle =0$.
Using $x^{+}(0)x^{-}{=}h$ this implies  
\begin{align}
2k\vac {=}h(1)(x^{+}(0)x^{-}){=}2x^{+}(1)x^{-},
\label{xp1xm}
\end{align}
by commutativity. 
 Hence, for $k\neq 0$, $\dlangle ~,~\drangle $ is a nondegenerate, invariant, bilinear form on ${\slt}$, so that for all $u,v\in \slt$
\begin{align}
\label{invform}
\dlangle  u,v \drangle {=}k(u,v ),
\end{align}
where here $(~,~)$ is 
the standard invariant bilinear form normalized to $(h,h)=2$. \eqref{invform} also holds for $k=0$. Hence, since $u(0)v=-v(0)u$ defines a commutator bracket on $ \slt$,  \eqref{KM} implies the desired result.
\end{proof}

\medskip
\begin{lem} Properties (a)--(h) above hold for $c=6k$.
\end{lem}
\begin{proof}
We set 
\begin{align}J^0{:=}h,~J^{\pm}:=x^{\pm}, ~G^{\pm}:=\tau^{\pm},\text{ and }\overline{G}^{\pm}:=\mp \overline{\tau}^{\pm}.
\end{align} 
By (\ref{sl2relation})--(\ref{xpmtau0}), (\ref{eq:kdef})--(\ref{xp1xm}), we have 
\begin{align*}
&\quad J^0(0)J^{\pm}=\pm 2J^{\pm}, \quad J^0(1)J^0=2k\vac,\\
& \quad J^+(0)J^-=J^0, \quad J^+(1)J^-=k\vac,\\
& \quad J^0(0)G^{\pm}=\pm G^{\pm},\quad J^0(0)\overline{G}^{\pm}=\pm\overline{G}^{\pm},\\
& \quad J^{\pm}(0)G^{\mp}=G^{\pm}, \quad J^{\pm}(0)\overline{G}^{\mp}=-\overline{G}^{\pm}.
\end{align*}
These are the needed relations. 
\end{proof}
\begin{rmk} In addition, we have 
$J^{\pm}(0)G^{\pm}=J^{\pm}(0)\overline{G}^{\pm}=0.$
\end{rmk}
\medskip
Next, we will obtain properties (i) and (j):
\begin{lem} We can choose the normalization for $\tau^{\pm},  \overline{\tau}^{\pm}$ so that
\begin{align}\label{hdef}
h{=}\tau^{\pm}(1)\overline{\tau}^{\mp},\quad 
\tau^{\pm}(1)\overline{\tau}^{\pm}{=}\mp 2x^{\pm}.
\end{align} In particular, we have $G^{\pm}(1)\overline{G}^{\pm}=2J^{\pm}$, and $G^{\pm}(1)\overline{G}^{\mp}=\pm J^0$.
\end{lem}
\begin{proof} Thanks to the $\slt$-morphism
$u{\otimes}v{\mapsto} u(1)v$ from $A{\otimes}B$ on to $A(1)B$, we get generators of $A(1)B$ just by replacing
$u{\otimes}v$ by $u(1)v$ in the generators for $\Sigma$ in (\ref{Sigmadef}).
Thus 
\begin{align*}
&A(1)B{=}\slt{=}\langle \tau^{+}(1)\overline{\tau}^{+}, \tau^{-}(1)\overline{\tau}^{-}, 
\tau^+(1)\overline{\tau}^-{+}\tau^-(1)\overline{\tau}^+\rangle.
\end{align*}
 Similarly, the generator for
$\Lambda$ in \eqref{Lambdadef} maps to $0$, leading to $\tau^+(1)\overline{\tau}^-{-}\tau^{-}(1)\overline{\tau}^+{=}0$.

\medskip
$\tau^+(1)\overline{\tau}^-{+}\tau^-(1)\overline{\tau}^+=2\tau^{\pm}(1)\overline{\tau}^{\mp}$ is a nonzero semisimple element of
$\slt$.\ As such, it is a nonzero multiple of $h$.\ Since we are free to simultaneously scale the $\tau$'s by a nonzero constant,  we may choose
the scale so that $h{=}\tau^{+}(1)\overline{\tau}^{-}=\tau^{-}(1)\overline{\tau}^{+}$.
Since $x^{\pm}(0)h=-h(0)x^{\pm}=\mp 2x^{\pm}$ we find, 
using super commutativity, that
\begin{align*}
- 2 x^{+}&= x^{+}(0) \left(\tau^{+}(1)\overline{\tau}^{-}\right)
\\
&=  \left(x^{+}(0)\tau^{+}\right)(1)\overline{\tau}^{-}+\tau^{+}(1)\left(x^{+}(0) \overline{\tau}^{-}\right)
\\
&= 0+\tau^{+}(1)  \overline{\tau}^{+}.
\end{align*}
Similarly, $\tau^{-}(1)  \overline{\tau}^{-} = 2 x^{-}$.
\end{proof}

\subsection{Assumption (III)}
Lemma~\ref{lem:AnpA} implies 
\begin{align}
&A(1)A{=}A(2)A{=}B(1)B{=}B(2)B{=}0.
\label{a(1)ab(1)b}
\end{align} 
Thus super commutativity implies
\begin{lem}
$[v(m),w(n)]=0$ for all $v,w\in A$ for all $v,w\in B$.
\end{lem}

\begin{lem}
\label{lem:u1v} 
$u(1)v{=}0$ for all $u{\in}\slt$ and $v{\in} A{\oplus}B$.\
\end{lem} 
\begin{proof} The vector space $\slt(1)A$ is spanned by the 6 vectors of the form $u(1)\tau^{\pm}$ for $u=x^{\pm}$ or $h$.
Recall that $ \tau^{+}(1)\overline{\tau}^{+}=-2x^{+}$ from \eqref{hdef}. Hence, by super associativity
\begin{align*}
x^{+}(1)\tau^{+} &=-\frac{1}{2}
 \sum_{i\geq 0}(-1)^i 
\binom{1}{i} \left(
\tau^{+}(1{-}i)\overline{\tau}^{+}(1{+}i){-}
\overline{\tau}^{+}(2{-}i){\tau}^{+}(i)
\right)\tau^{+}\\
&=-\frac{1}{2}\tau^{+}(1)\overline{\tau}^{+}(1)\tau^{+}+0=-\tau^{+}(1)x^{+},
\end{align*}
using \eqref{a(1)ab(1)b}, $\overline{\tau}^{+}(2)\tau^{+}\in\CC\vac$ and that $
\overline{\tau}^{+}(1)\tau^{+}=-\tau^{+}(1)\overline{\tau}^{+}=2x^{+}$  by super skew symmetry.
But, by  super skew symmetry again we directly have  $x^{+}(1)\tau^{+}=+\tau^{+}(1)x^{+}$. Thus $x^{+}(1)\tau^{+}=0$. Furthermore, $(x^{-}(0))^kx^{+}(1)\tau^{+}=0$  for $k=1,2,3$ results in the identities
\begin{align*}
h(1)\tau^{+}-x^{+}(1)\tau^{-}=x^{-}(1)\tau^{+}+h(1)\tau^{-}=x^{-}(1)\tau^{-}=0.
\end{align*}
 We can then repeat a similar argument based on super associativity and skewsymmetry using $h=\tau^{-}(1)\overline{\tau}^{+}$ to show that $h(1)\tau^{\pm}=0$. Thus $\slt(1)A=0$. By an identical argument we also find $\slt(1)B=0$.
\end{proof}

We next establish property (k):
\begin{lem} 
\label{lem:tau2tau} We have
\begin{align*}
\tau^{\pm}(2)  \overline{\tau}^{\pm}=0,\ \ \tau^{\pm}(2)  \overline{\tau}^{\mp}=\pm 2k\vac.
\end{align*}
Consequently, for $c=6k$ we have $$G^{\pm}(2)\overline{G}^{\pm}=0,\text{ and }G^{\pm}(2)\overline{G}^{\mp}{=}\frac{c}{3}\vac.$$

\end{lem} 
\begin{proof}
The image of $\Sigma{\rightarrow}A(2)B{=}\CC\vac $ is $\slt$-invariant, hence is $0$.\ Then
\begin{align*}
\tau^{\pm}(2)\overline{\tau}^{\pm}{=} 
\tau^{+}(2)\overline{\tau}^{-}{+}\tau^{-}(2)\overline{\tau}^{+}=0.
\end{align*} 
The invariant $\tau^{+}(2)\overline{\tau}^{-}{-}\tau^{-}(2)\overline{\tau}^{+}$ is computed as follows{:}\
$h(1) \tau^{\pm}{=}h(1)\overline{\tau}^{\pm}{=}0$ (Lemma \ref{lem:u1v}) together with  commutativity imply
\begin{align}
[h(1),\tau^{\pm}(n)]{=}{\pm} \tau^{\pm}(n+1),\quad [h(1),\overline{\tau}^{\pm}(n)]{=}{\pm} \overline{\tau}^{\pm}(n+1).
\label{h1taun}
\end{align}
Using \eqref{hdef} we thus  find  
\begin{align*}
2k\vac=h(1)h &=h(1) \left(\tau^{\pm}(1)\overline{\tau}^{\mp}\right)=\pm\tau^{\pm}(2)\overline{\tau}^{\mp}.
\end{align*}
\end{proof}

\begin{lem}\label{lemsigma} Let 
\begin{align*}
\sigma{:=}\frac{1}{2}(\tau^+(0)\overline{\tau}^-{-}\tau^{-}(0)\overline{\tau}^+).
\end{align*}
Then $\sigma(0)u=Tu$, $\sigma(1)u{=}u$ and $\sigma(2)u{=}0$ for all $u{\in}\slt$.\ In particular, $\sigma$ is a \emph{nonzero} $\slt$-invariant.
\end{lem}

\begin{proof} $\sigma$ is an $\slt$-invariant in $A(0)B\subseteq V_{2}$ because it generates the image
of $\Lambda$ under the $\slt$-morphism defined by $A{\otimes}B{\rightarrow} A(0)B$. 

$\sigma(2)\slt\in \CC \vac$ implies $\sigma(2)\slt=0$ since $\CC \vac$ is a trivial $\slt$ representation. 
By skewsymmetry, $\sigma(1)u{=}u(1)\sigma$ for all $u{\in }\slt$. Hence, using \eqref{h1taun} and Lemma~\ref{lem:u1v}, we find
\begin{align*}
\sigma(1)h&=\frac{1}{2}h(1)\left(\tau^+(0)\overline{\tau}^-{-}\tau^{-}(0)\overline{\tau}^+\right)\\
&=\frac{1}{2}\left(\tau^+(1)\overline{\tau}^-{+}\tau^-(1)\overline{\tau}^+\right){=}h,
\end{align*}
 from  \eqref{hdef}.\ This proves, in particular, that $\sigma{\not=}0$.
Since $\sigma$ is $\slt$-invariant and $h$ generates $A(1)B$ as an $\slt$--module, it follows that
$\sigma(1)u{=}u$ for all $u{\in}\slt$. Furthermore, from skew-symmetry we find that
\begin{align*}
\sigma(0)u &= -u(0)\sigma +T(u(1)\sigma)\\
&= 0+T(\sigma(1)u)=Tu,
\end{align*}
using the $\slt$-invariance of $\sigma$. 
\end{proof}

\begin{lem}\label{lemsigmatau} 
$\sigma(1)v {=}\frac{3}{2}v $ and $\sigma(2)v=0$ for all $v{\in}A\oplus B$.
\end{lem}
\begin{proof}
Using superassociativity we find
\begin{align*}
\sigma(1)=\frac{1}{2}[\tau^{+}(0), \overline{\tau}^{-}(1)]
-\frac{1}{2}[\tau^{-}(0), \overline{\tau}^{+}(1)].
\end{align*}
Hence 
\begin{align*}
\sigma(1)\tau^{+}&=\frac{1}{2}\tau^{+}(0)\overline{\tau}^{-}(1)\tau^{+}
-\frac{1}{2}\tau^{-}(0)\overline{\tau}^{+}(1)\tau^{+}\\
&=-\frac{1}{2}\tau^{+}(0)h-\frac{1}{2}\tau^{-}(0)(2x^{+})\\
&=\frac{1}{2}\tau^{+}+\tau^{+}=\frac{3}{2}\tau^{+},
\end{align*}
using $\tau^{\pm}(0)\tau^{+}=0$  (by  assumption (III), \eqref{hdef} and super skew-symmetry).\ 
Using superassociativity, we similarly find
\begin{align*}
\sigma(2)=\frac{1}{2}[\tau^{+}(0), \overline{\tau}^{-}(2)]
-\frac{1}{2}[\tau^{-}(0), \overline{\tau}^{+}(2)].
\end{align*}
Hence 
\begin{align*}
\sigma(2)\tau^{+}&=\frac{1}{2}\tau^{+}(0)\overline{\tau}^{-}(2)\tau^{+}
-\frac{1}{2}\tau^{-}(0)\overline{\tau}^{+}(2)\tau^{+}=0.
\end{align*}
Similar results follow for $\tau^{-}$ by $\slt$-symmetry and for $\overline{\tau}^{\pm}$. 
\end{proof}

\subsection{Assumption (IV)}
We next  obtain properties (l) and (m).
\begin{lem}\label{lem:tau0tau}
We have
\begin{align*}
\tau^{\pm}(0)\overline{\tau}^{\pm}{=}\mp Tx^{\pm},\ \  
\tau^{+}(0)\overline{\tau}^{-}+\tau^{-}(0)\overline{\tau}^{+}{=}Th.
\end{align*}
 Consequently,
\begin{align*}
&G^{\pm}(0)\overline{G}^{\pm}{=}TJ^{\pm},\text{ and }G^{\pm}(0)\overline{G}^{\mp}{=}\sigma\pm \frac{1}{2}TJ^0.
\end{align*}
\end{lem}
\begin{proof}
By assumption (IV),  there exists $u\in \slt$ such that $Tu\in A(0)B$. 
Since $A(0)B$ is an $\slt$-module and $T\slt{\cong}\slt$ it follows that $T\slt{\subseteq}A(0)B$. 
Thus there is a morphism of $\slt$ adjoint-modules
$\Sigma{\rightarrow} T\slt{\subseteq} A(0)B$. In particular
\begin{align*}
\tau^{-}(0)\overline{\tau}^{-} {=} \kappa T x^{-},
\end{align*}
for some  $\kappa\neq 0$. Using \eqref{eq:Lvn} we have
\begin{align*}
L^{U}(1)\tau^{-}(0)\overline{\tau}^{-}{=}\tau^{-}(1)\overline{\tau}^{-}=2x^{-},
\end{align*}
from \eqref{hdef}. This implies
\[
2x^{-}{=}L^{U}(1)\kappa L^{U}(-1)x^{-}{=}2\kappa x^{-}.
\]
Hence $\kappa{=}1$ and $\tau^{-}(0)\overline{\tau}^{-} {=}  T x^{-}$. The other relations   follow by $\slt$-symmetry. 
\end{proof}

\begin{lem} 
\label{lem:vcom}
The modes of $\tau^{\pm},\overline{\tau}^{\pm}$ satisfy the commutator relations
\begin{align}
[\tau^{\pm}(m),\overline{\tau}^{\pm}(n)] &=\pm (n-m)x^{\pm}(m+n-1),
\label{eq:tptp}\\
[\tau^{\pm}(m),\overline{\tau}^{\mp}(n)] &=\pm \sigma(m+n)
+\frac{1}{2}(m-n)h(m+n-1)
\pm m(m-1)k\delta_{m+n+1,0}.
\label{eq:tptm}
\end{align}
 \end{lem}
\begin{proof} 
Properties (k)--(m), which we have already established, show that
\begin{align*}
& \tau^{\pm}(0)\overline{\tau}^{\pm}{=}\mp Tx^{\pm},\quad 
\tau^{\pm}(1)\overline{\tau}^{\pm}{=}\mp 2x^{\pm},\quad
\tau^{\pm}(2)\overline{\tau}^{\pm}{=}0,\\
& \tau^{\pm}(0)\overline{\tau}^{\mp}{=}\frac{1}{2} Th \pm\sigma,\quad 
\tau^{\pm}(1)\overline{\tau}^{\mp}{=}h,\quad
\tau^{\pm}(2)\overline{\tau}^{\mp}{=}\pm 2k\vac,
\end{align*}
which imply the result  from the super commutator formula.
\end{proof}

 \begin{lem} 
\label{lem:sig0tau}
$\sigma(0)v {=}Tv $ for all $v{\in}A\oplus B$.
 \end{lem}
\begin{proof} Using \eqref{eq:tptm} we have
\begin{align*}
\sigma(0)=[\tau^{+}(0),\overline{\tau}^{-}(0)].
\end{align*}
Thus we find using super skew-symmetry that
\begin{align*}
\sigma(0)\tau^{+}&{=}\tau^{+}(0)\overline{\tau}^{-}(0) \tau^{+}{+}0\\
&{=}\tau^{+}(0)\left(\tau^{+}(0)\overline{\tau}^{-} {-}T(\tau^{+}(1)\overline{\tau}^{-})\right)\\
&{=} 0 {-}\tau^{+}(0)Th\\
&{=}(Th)(0)\tau^{+}{-}T((Th)(1)\tau^{+})+0\\
&{=} 0{+}T\tau^{+},
\end{align*}
using $(Th)(0)=0$ and $(Th)(1)=-h(0)$. A similar argument applies to the remaining elements of $A\oplus B$.
\end{proof}

 \begin{lem} \label{lem: siguv}
For all $u\in A(1)B$ and $v\in A\oplus B$ we have
\begin{align}
[\sigma(m),u(n)] 
&{=}-nu(m+n-1),
\label{eq:sigun}
\\
[\sigma(m),v(n)] 
&
{=}\left(\frac{1}{2}m-n\right)v(m+n-1).
\label{eq:sigvn}
\end{align}
 \end{lem}
\begin{proof}
Using  Lemma~\ref{lemsigma} we have $\sigma(2)u{=}0$, $\sigma(1)u{=}u$ and $\sigma(0)u{=}Tu$ for $u\in \slt$ the result \eqref{eq:sigun} follows from the commutator formula. 
Likewise, \eqref{eq:sigvn} follows from $\sigma(2)v{=}0$, $\sigma(1)v{=}\frac{3}{2}v$ and $\sigma(0)v{=}Tv$ for all $v\in A{\oplus} B$ using Lemmas~\ref{lemsigmatau} and \ref{lem:sig0tau}.
\end{proof}

 \begin{lem}\label{lemvir} $\sigma$ is a Virasoro vector for central charge $c{=}6k$.
 \end{lem}
\begin{proof} We have to check the relations 
\begin{align*}
&\sigma(0)\sigma{=}T\sigma,\quad
\sigma(1)\sigma{=}2\sigma,\quad
\sigma(2)\sigma{=}0,\quad
\sigma(3)\sigma{=}3k\vac.
\end{align*}
Using \eqref{eq:sigvn} we find
\begin{align*}
\sigma(1)\sigma &=\frac{1}{2}\left(\sigma(1)\tau^{+}(0)\overline{\tau}^{-}-\sigma(1)\tau^{-}(0)\overline{\tau}^{+}\right)
\notag
\\
&=\frac{1}{2}\left(\frac{1}{2}\tau^{+}(0)\overline{\tau}^{-}+\tau^{+}(0)\frac{3}{2}\overline{\tau}^{-}
-\frac{1}{2}\tau^{-}(0)\overline{\tau}^{+}-\tau^{-}(0)\frac{3}{2}\overline{\tau}^{+}\right)
= 2\sigma,\\
\sigma(2)\sigma &=\frac{1}{2}\left(\sigma(2)\tau^{+}(0)\overline{\tau}^{-}-\sigma(2)\tau^{-}(0)\overline{\tau}^{+}\right)
\notag
\\
&=\frac{1}{2}\left(\tau^{+}(1)\overline{\tau}^{-}+0
-\tau^{-}(1)\overline{\tau}^{+}-0\right)=0,
\\
\sigma(3)\sigma &=\frac{1}{2}\left(\sigma(3)\tau^{+}(0)\overline{\tau}^{-}-\sigma(3)\tau^{-}(0)\overline{\tau}^{+}\right)
\notag
\\
&=\frac{1}{2}\left(\frac{3}{2}\tau^{+}(2)\overline{\tau}^{-}+0
-\frac{3}{2}\tau^{-}(2)\overline{\tau}^{+}-0\right)
=3k\vac.
\end{align*} 
Lastly, by skew symmetry
\begin{align*}
\sigma(0)\sigma &=-\sigma(0)\sigma+T(\sigma(1)\sigma)
\notag
\\
&=-\sigma(0)\sigma+2T\sigma,
\end{align*}
so that $\sigma(0)\sigma =T\sigma$.
\end{proof}

Thus Theorem~\ref{thmUN=4} holds since all the defining relations for the $N{=}4, c{=}6k$ super conformal algebra $\Nfour$ are satisfied.

\begin{rmk}
If  $U$ is $C_2$--cofinite and of strong CFT type then $k$ is a positive integer by \cite{DM}  since $\slt{\subseteq} U_1$. We also note that the $N{=}4$ Virasoro vector $\sigma$ (of Lemma~\ref{lemsigma}) and $\omega^U$ (the Virasoro element of $U$) can be independent vectors of different central charges. 
Thus Theorem~\ref{thmUN=4} is not a generating theorem for $N{=}4$ algebras (such as in \cite{K} or \cite{dS}) but rather describes the existence of a $N{=}4$ subalgebra  of a given SVOA.
\end{rmk}
\medskip

Finally we note that the automorphism group of  $\Nfour$ contains an involution $g$ defined by
\begin{align*}
g:A\oplus B &\rightarrow B\oplus A\\
\left(\tau^{\pm},\overline{\tau}^{\pm}\right) &\mapsto \left(\overline{\tau}^{\pm},-\tau^{\pm}\right).
\end{align*}
This  follows by directly verifying that the defining relations (a)--(m) are preserved by $g$ by use of super skew--symmetry i.e.  $u(1)v=v(1)u$ and $u(0)v-v(0)u=-T(u(1)v)\in T\slt$ for all $u\in A$ and $v\in B$. Furthermore, Assumption~(IV) of Theorem~\ref{thmUN=4} therefore has the following reformulation
\begin{lem}
\label{lem:AB=BA}
$T\slt\cap\ A(0)B\neq 0 $ if and only if $A(0)B=B(0)A$.\hfill $\Box$
\end{lem}


\section{Proof of Theorem \ref{thmVLV}}\label{SRecog2}

In this Subsection we assume the hypotheses and notation of Theorem \ref{thmVLV}, in particular 
$V$ is contained in a super lattice VOA $V_L$ (see Subsection~\ref{VL} for the definition and relevant properties).\
In particular, $(V_L)_1$ is a reductive Lie algebra and each of its  components is a simple Lie algebra of type
$ADE$ and of level $1$.

\medskip
Now by hypothesis (IV) of Theorem \ref{thmVLV} our Lie algebra $\slt$ contains a root of $L$.\ It follows that
$\slt$ is contained in one of the components of $V_1$ and therefore it also has level $k{=}1$.\ Adopting the notation of the previous Section, it follows that $(h, h){=}2$ (cf. Lemma \ref{lemKM}).\ Thus $h$ is a root of
$L$ and we have $\slt{=}\langle h, e^{\pm h}\rangle$.

\medskip
We will deduce Theorem \ref{thmVLV} from Theorem \ref{thmUN=4}.\ To this end, notice that states of weight $\frac{3}{2}$ in $V_L$ are primary because
$L$ has no vectors of norm $1$.\ Thus it suffices to take $U{:=}V_L$ in Theorem \ref{thmUN=4} and show that hypotheses (I)--(IV) of Theorem \ref{thmUN=4} hold.\ Then Theorem \ref{thmUN=4} shows that  $V$ is the $N{=}4$ super conformal algebra of central charge $6k{=}6$. \ Parts (I) and (II) hold by assumption, so we only have to establish (III) and (IV).

\medskip
We need some additional notation.\ Let $(\ ,\ ){:}L{\times}L{\rightarrow}\ZZ$ be the bilinear form on $L$.\ Note that
this is \emph{not} the notation used in the proof of Lemma \ref{lemKM}, where $(\ ,\ )$ denoted the invariant bilinear form on $\slt$.\ 

\medskip
The vectors in $L$ of norm $n$ are denoted by
\begin{align*}
L_n{:=}\{\alpha{\in}L{\mid}(\alpha, \alpha){=}n\},
\end{align*}
in particular $L_2$ is the root system of $L$.\ Fix a multiplicative bicharacter
$\varepsilon{:}L{\times}L{\rightarrow}\{\pm 1\}$ that defines the central extension $\widehat{L}$ occurring in the short exact sequence
\begin{align*}
1\rightarrow \{\pm 1\}\rightarrow \widehat{L}\rightarrow L \rightarrow 0.
\end{align*}
See Subsection \ref{SSepsilon} for more details on the $\varepsilon$-formalism, in particular
for the justification that
\begin{eqnarray}\label{epsvalue}
\varepsilon(\alpha, \alpha)=\varepsilon(\alpha, -\alpha){=}  \left \{ \begin{array}{ll}
{-}1 & \alpha{\in}L_2 \\
\; \; \, 1 & \alpha{\in}L_3{\cup}L_4
\end{array} \right.
\end{eqnarray}

\medskip
Recall (\ref{ABnames}) that $\tau^+, \overline{\tau}^+$ are highest weight vectors in $A$ and $B$ respectively.\ We have already mentioned that $U_{\frac{3}{2}}$ is spanned by states $e^{\beta}\ (\beta{\in}L_3)$. Thus there are
nonempty subsets $X, Y{\subseteq}L_3$ and scalars $c_{\alpha}, d_{\lambda}$ such that
\begin{align}\label{tauvec}
\tau^+{:=}\sum_{\alpha\in X}c_{\alpha}e^{\alpha},\ \ \overline{\tau}^+{:=}\sum_{\lambda\in Y}d_{\lambda}e^{\lambda}.
\end{align}
Because
$h(0)\tau^+{=}\tau^+$ then we have $(h, \alpha){=}1\ (\alpha{\in}X)$, and similarly
$(h, \lambda){=}1\ (\lambda{\in}Y)$.

\medskip
As a result, we have the following useful facts that hold for
$\alpha, \beta{\in}X$.\ $|(\alpha, \beta)|{\leq}3$ by the Schwarz inequality, moreover
\begin{align*}
(\alpha, \beta){=} 
\begin{cases}
{\pm}3 &\mbox{iff}\  \alpha{=}\pm\beta \\
\;\;\,2 &\mbox{iff}\ \alpha{-}\beta{=}\gamma\  (\mbox{root}\ \gamma\perp h)\\
{-}2 &\mbox{iff}\ \alpha{+}\beta{=}h \\
{-}1 &\mbox{iff}\ \alpha{+}\beta{=}h{+}\gamma\ (\mbox{root}\ \gamma\perp h)
\end{cases}
\end{align*}
Identical formulas hold in case $\alpha, \beta{\in}Y$.\
We use these formulas in later calculations.

\medskip
The elements $x^{\pm}{\in}\slt$
may be identified (recall that $\varepsilon(h,h)=-1$ by (\ref{epsvalue})) as 
 $x^{\pm}=\mp e^{\pm h}$.\
Because {$x^-(0)\tau^+{=}\tau^-$} we have
\begin{align*}
\tau^-{=}\sum_{\alpha\in X} c_{\alpha}e^{{-}h}(0)e^{\alpha}{=}\sum_{\alpha\in X} c_{\alpha}\varepsilon(h, \alpha)e^{\alpha-h},
\end{align*}
and similarly
\begin{align*}
\overline{\tau}^-{=}\sum_{\lambda\in Y} d_{\lambda}\varepsilon(h, \lambda)e^{\lambda-h}.
\end{align*}

We now consider the consequences of assumption (III) of Theorem~\ref{thmVLV}. 
\begin{lem}\label{lemA(1)A} $A(1)A=0$ implies
\begin{align*}
&(a)\quad \sum_{\alpha\in X} c_{\alpha}c_{h+\gamma-\alpha}\varepsilon(\alpha, h{+}\gamma){=}0\ \ (\mbox{each root
$\gamma\perp h$}),\\
&(b)\quad \sum_{\alpha\in X} c_{\alpha}c_{h-\alpha}\varepsilon(h, \alpha)\alpha{=}0,
\end{align*}
with a corresponding statement concerning $B(1)B=0$.
\end{lem}
\begin{proof}
$A(1)A$ contains the element $\tau^+(1)\tau^-$, which is equal to
\begin{align*}
&\sum_{\alpha, \beta\in X} c_{\alpha}c_{\beta}\varepsilon(h, \beta)e^{\alpha}(1)e^{\beta-h}\;{=}
\sum_{(\alpha, \beta-h)=-2, -3} c_{\alpha}c_{\beta}\varepsilon(h, \beta)e^{\alpha}(1)e^{\beta-h}\\
&=\sum_{\alpha+\beta=h} c_{\alpha}c_{\beta}\varepsilon(h, \beta)\alpha
+\sum_{(\alpha, \beta)=-1} c_{\alpha}c_{\beta}\varepsilon(h, \beta)\varepsilon(\alpha, \beta{-}h)e^{\alpha+\beta-h}\\
&=\sum_{\alpha+\beta=h} c_{\alpha}c_{\beta}\varepsilon(h, \beta)\alpha
+\sum_{\gamma\perp h}\left\{\sum_{\alpha+\beta=h+\gamma} c_{\alpha}c_{\beta}\varepsilon(h, \beta)\varepsilon(\alpha, \gamma{-}\alpha)\right\}e^{\gamma}\\
&=\sum_{\alpha+\beta=h} c_{\alpha}c_{\beta}\varepsilon(h, \beta)\alpha
-\sum_{\gamma\perp h}\sum_{\alpha+\beta=h+\gamma} c_{\alpha}c_{\beta}\varepsilon(h,\gamma{+}\alpha)\varepsilon(\alpha, \gamma)e^{\gamma}\\
&=\sum_{\alpha+\beta=h} c_{\alpha}c_{\beta}\varepsilon(h, \beta)\alpha
-\sum_{\gamma\perp h}\varepsilon(h, \gamma)\sum_{\alpha+\beta=h+\gamma} c_{\alpha}c_{\beta}\varepsilon(h,\alpha)\varepsilon(\alpha, \gamma)e^{\gamma}\\
&=-\sum_{\alpha} c_{\alpha}c_{h-\alpha}\varepsilon(h, \alpha)\alpha
+\sum_{\gamma\perp h}\varepsilon(h, \gamma)\sum_{\alpha} c_{\alpha}c_{h+\gamma-\alpha}\varepsilon(\alpha,h{+}\gamma)e^{\gamma},
\end{align*}
using $\varepsilon(h, \beta){=}\varepsilon(h, h{-}\alpha){=}{-}\varepsilon(h, \alpha)$.
Hence $A(1)A{=}0$ if and only if
\begin{align*}
&\sum_{\alpha\in X} c_{\alpha}c_{h+\gamma-\alpha}\varepsilon(\alpha, h{+}\gamma){=}0,
\end{align*}
for each root $\gamma\perp h$, and
\begin{align*}
\sum_{\alpha} c_{\alpha}c_{h-\alpha}\varepsilon(h, \alpha)\alpha=0.
\end{align*}
\ A similar analysis applies for $B(1)B=0$.
\end{proof}

\begin{lem} We have $A(0)A{=}B(0)B{=}0$.
\end{lem}
\begin{proof} We prove that $A(0)A{=}0$.\ The proof that $B(0)B{=}0$ is similar.\ Assume, then, that
$A(0)A{\not=0}$.\ Because $A(1)A{=}0$, by assumption~(III),
then $A(0)A$ has dimension ${\leq} 3$ by super skew-symmetry, indeed because we are assuming that $A(0)A{\not=}0$
then the image of $A{\otimes}A{\rightarrow}A(0)A$ is the adjoint module for $\slt$.\
Now it follows that $0{\not=}\tau^+(0)\tau^+{\in}A(0)A$.
But we also have
\begin{align*}
&\tau^+(0)\tau^+{=}\sum_{\alpha, \beta\in X}c_{\alpha}c_{\beta}e^{\alpha}(0)e^{\beta}
{=}\sum_{(\alpha, \beta)=-1}c_{\alpha}c_{\beta}\varepsilon(\alpha, \beta)e^{\alpha{+}\beta}
{+}\sum_{(\alpha, \beta)=-2}c_{\alpha}c_{\beta}\varepsilon(\alpha, \beta)\alpha(-1)e^{\alpha+\beta}\\
{=}&\sum_{\gamma}\left\{\sum_{\alpha+\beta=h+\gamma}c_{\alpha}c_{\beta}\varepsilon(\alpha, h{+}\gamma)\right\}
e^{h+\gamma}
{+}\left\{\sum_{\alpha+\beta=h}c_{\alpha}c_{\beta}\varepsilon(\alpha, h)\alpha\right\}(-1)e^{h}{=}0,
\end{align*}
where we used Lemma \ref{lemA(1)A}.\ This contradiction completes the proof of
the Lemma.
\end{proof}
\noindent This establishes assumption (III) of Theorem \ref{thmUN=4}.
\medskip

We now consider consequences of  assumptions~(II)  and (IV) of Theorem~\ref{thmVLV}.
\begin{lem}\label{lemP} $A(1)B\cong \slt$  with $\slt=\langle h,e^{\pm h}\rangle $ if and only if
\begin{align*}
&(a)\quad h=-\sum_{\alpha\in X}c_{\alpha}d_{h-\alpha}\varepsilon(h, \alpha)\alpha,
\\
&(b)\quad  \sum_{\gamma\perp h}\varepsilon(h, \gamma)\left\{\sum_{\alpha\in X} c_{\alpha}d_{h+\gamma-\alpha}
\varepsilon(\alpha, h{+}\gamma)\right\}e^{\gamma}=0,
\end{align*}
where the $\gamma$ sum is taken  over each root $\gamma\perp h$.
\end{lem}
\begin{proof} By calculations similar to those of Lemma \ref{lemA(1)A}
we have
\begin{align*}
\tau^+(1)\overline{\tau}^{-}=&
\sum_{\alpha\in X, \lambda\in Y} c_{\alpha}d_{\lambda}\varepsilon(h, \lambda)e^{\alpha}(1)e^{\lambda-h}\; {=}
\sum_{(\alpha, \lambda-h)=-2, -3} c_{\alpha}d_{\lambda}\varepsilon(h, \lambda) e^{\alpha}(1)e^{\lambda-h}
\\
=&-\sum_{\alpha\in X} c_{\alpha}d_{h-\alpha}\varepsilon(h, \alpha)\alpha
+\sum_{\gamma\perp h}\varepsilon(h, \gamma)\sum_{\alpha\in X} c_{\alpha}d_{h+\gamma-\alpha}\varepsilon(\alpha,h{+}\gamma)e^{\gamma}.
\end{align*}
But from  \eqref{hdef}  we have that $h=\tau^{+}(1)\overline{\tau}^{-}$ iff (a) and (b) hold. Note that taking the inner product of (a) with $h$ implies 
\begin{align}
\sum_{\alpha\in X}c_{\alpha}d_{h-\alpha}\varepsilon(h, \alpha)=-2.
\label{eq:2sum}
\end{align}
This implies 
\begin{align*}
\tau^{-}(1)\overline{\tau}^{+}=&\sum_{\alpha\in X, \lambda\in Y} c_{\alpha}d_{\lambda}\varepsilon(h,\alpha)e^{\alpha-h}(1)e^{\lambda}\; {=}
\sum_{(\alpha-h, \lambda)=-2, -3} c_{\alpha}d_{\lambda} \varepsilon(h,\alpha)e^{\alpha-h}(1)e^{\lambda}\\
=&\sum_{\alpha\in X} c_{\alpha}d_{h-\alpha} \varepsilon(h,\alpha)(\alpha-h)
+\sum_{\gamma\perp h}\varepsilon(h, \gamma)\sum_{\alpha\in X} c_{\alpha}d_{h+\gamma-\alpha}\varepsilon(\alpha,h{+}\gamma)e^{\gamma}\\
=& -h+2h+0=h,
\end{align*} 
 iff (a) and  (b) hold. The remaining relations  in  \eqref{hdef} follow from $\slt$ symmetry.
\end{proof}

\begin{lem} $A(1)B\cong \slt$ implies $T\slt\cap A(0)B{\not=}0$.
\end{lem}
\begin{proof} $A(0)B$ contains the element 
\begin{align*}
\tau^+(0)\overline{\tau}^+&{=}\sum_{(\alpha, \lambda)=-1, -2}c_{\alpha}d_{\lambda}e^{\alpha}(0)e^{\lambda}\\
&{=}-\frac{1}{2}\left\{\sum_{\alpha\in X}c_{\alpha}d_{h-\alpha}\varepsilon(h, \alpha)\right\}h(-1)e^h{=}h(-1)e^h=Te^{-h},
\end{align*}
by \eqref{eq:2sum} of Lemma \ref{lemP}.\
Thus $T\slt\cap A(0)B{\not=}0$.
\end{proof}

\medskip
This completes the proof of hypothesis(IV) of Theorem \ref{thmUN=4}, and with it the proof of Theorem \ref{thmVLV}.


\section{$N{=}2$ superconformal algebras}\label{SN=2}
The $N{=}2$ SVOA $\Ntwo$ of central charge $c$  is
generated by a pair of states $\tau^{\pm}$ of conformal weight $\frac{3}{2} $ satisfying the non-zero relations e.g. \cite{K}
\begin{align*} 
&\mbox{(i)}\; && \tau^{\pm}(2)\tau^{\mp}{=}\frac{c}{3}\vac,\quad &
&\mbox{(ii)}\; && \tau^{\pm}(1)\tau^{\mp}{=}\pm h, \quad &
&\mbox{(iii)}\; && \tau^{\pm}(0)\tau^{\mp}{=}\omega{\pm} \frac{1}{2}T,&
\\
&\mbox{(iv)}\; && h(0)\tau^{\pm}{=}\pm\tau^{\pm},\quad &
&\mbox{(v)}\; && h(1)h{=}\frac{c}{3}\vac,& &
\end{align*}
together with the standard Virasoro relations between $\tau^{\pm},h$ and the 
Virasoro vector $\omega$ of central charge $c$. 
Similarly to Theorem~\ref{thmUN=4} we have
\begin{thm}\label{thmUN=2} 
Let $U$ be a SVOA of CFT--type \ 
Let  $V{\subseteq} U$ be the subalgebra generated by 2 \emph{primary vectors} $\tau^{\pm}$ of weight $\frac{3}{2} $ in $U$, so that 
\begin{align*}
V{=}\CC\mathbf{1} {\oplus} V_{\half} {\oplus} V_1{\oplus} V_{\frac{3}{2} } {\oplus}\hdots
\end{align*}
is a conformally graded subspace of $U$.\ Assume that:
\begin{enumerate}
	\item[(I)] $h(0)\tau^{\pm}{=}\pm \tau^{\pm}$ where $h{:=}\tau^{+}(1)\tau^{-}$,
	\item[(II)] $\tau^{\pm}(0)\tau^{\pm}{=}0$.
\end{enumerate}
Then $V\cong \Ntwo$ with central charge $c{=}6k$ where $h(1)h=2k\vac$.
\end{thm}
\begin{proof}
We sketch the proof, which is similar in many respects to that for Theorem~\ref{thmUN=4}.
We firstly note that $u(1)v{=}{-}v(1)u$ for all $u,v \in \langle\tau^{\pm} \rangle$ by super skew-symmetry. 
Thus Assumptions~(I) and (II) above imply the properties~(ii), (iv) and (v) for $c{=}6k$.\
As for Lemma~\ref{lem:u1v}, we find $h(1)\tau^{\pm}=0$ which in turn implies (as in Lemma~\ref{lem:tau2tau}) that $\tau^{\pm}(2)\tau^{\mp}{=}2k\vac $.\ Thus property~(i) holds.

\medskip
Using super skew-symmetry we have 
\begin{align*}
\tau^{+}(0)\tau^{-}{=}\tau^{-}(0)\tau^{+}{-}T(\tau^{-}(1)\tau^{+}){=}\tau^{-}(0)\tau^{+}{+}Th.
\end{align*}
Thus, in this case, we define
\begin{align}
\sigma{:=}\frac{1}{2}\left(\tau^{+}(0)\tau^{-}{+}\tau^{-}(0)\tau^{+}\right),
\label{eq:N2Vir}
\end{align}
so that $\tau^{\pm}(0)\tau^{\mp}{=}\sigma{\pm} \frac{1}{2}Th$. 
Hence
\begin{align}
[\tau^{+}(m), \tau^{-}(n)] &=\sigma(m+n) {+}\frac{1}{2}(m-n)h(m{+}n{-}1)
{+}m(m{-}1)k\delta_{m{+}n{+}1,0}.
\label{eq:taupm}
\end{align}

It remains to show that $\sigma$ is a Virasoro vector of central charge $c{=}6k$. \
As in Lemma~\ref{lemsigma}, we find $\sigma(0)h{=}Th$, $\sigma(1)h{=}h$ and $\sigma(2)h{=}0$. 
\eqref{eq:taupm} implies
\begin{align*}
\sigma(0){=}[\tau^{+}(0),\tau^{-}(0)],\quad 
\sigma(1){=}[\tau^{+}(0),\tau^{-}(1)]{+}\frac{1}{2}h(0),\quad
\sigma(2){=}[\tau^{+}(0),\tau^{-}(2)]+ h(1),
\end{align*}  
from which it follows that $\sigma(0)\tau^{+}=T\tau^{+}$, $\sigma(1)\tau^{+}=\frac{3}{2}\tau^{+}$  and $\sigma(2)\tau^{+}=0$  (cf. Lemmas~\ref{lemsigmatau} and \ref{lem:sig0tau}). Similar results follow for $\sigma(n)\tau^{-}$ for $n=0,1,2$. Hence (cf.  Lemma~\ref{lem: siguv})
\begin{align*}
[\sigma(m),\tau^{\pm}(n)] {=}\left(\frac{1}{2}m{-}n\right)\tau^{\pm}(m{+}n{-}1),
\end{align*} 
which implies that $\sigma$ is a Virasoro vector of central charge $6k$ (cf. Lemma~\ref{lemvir}).
\end{proof}

\section{Examples}\label{SExample}
 We provide some  constructions of $N{=}4$ and $N{=}2$ subalgebras of a lattice SVOA $V_L$ for an odd lattice $L$. 
These examples illustrate Theorems \ref{thmUN=4}, \ref{thmVLV} and \ref{thmUN=2}. 
Throughout, we let $L_n$ denote the set of lattice vectors in $L$ of norm $n$.

\medskip

\subsection{Example 1.} Consider the lattice SVOA $V_L$ for $L{=}\ZZ^6$ -- the well-known rank $12$ free fermion construction.\ Let $L{=}{\ZZ^6}$ be generated by $\gamma_1,\cdots,\gamma_6\in L_{1}$ with $(\gamma_i,\gamma_j){=}\delta_{ij}$.\ Then $V_{L}$ is generated by $12$ weight $\frac{1}{2}$  fermion vectors $e^{\pm \gamma_i}$.

\medskip
We firstly note from Section~\ref{SSepsilon} that 
\begin{align*}
\varepsilon(\gamma_i,\gamma_j)=
\begin{cases}
-\varepsilon(\gamma_j,\gamma_i)  &\mbox{ for } i\neq j,\\
-1  &\mbox{ for } i= j.
\end{cases}
\end{align*}
In addition, for convenience, we choose $\varepsilon(\gamma_1,\gamma_2)=1$ so that  
$\varepsilon(\gamma_2,\gamma_1)=-1$.

 \medskip
Define $\slt$ generators $h{:=}\gamma_1{+}\gamma_2,\, x^{\pm}{:=}\mp e^{\pm h}\in (V_{\ZZ^6})_{1}$ where 
$h(1)h{=}2\vac$ i.e. $k{=}1$ in \eqref{invform}.
Then $\langle e^{\gamma_1},e^{-\gamma_2}\rangle$  and 
$\langle e^{\gamma_2},e^{-\gamma_1}\rangle$ form a pair of $\slt$--representations, where using Subsection~\ref{VL}, we find
\begin{align}
h(0)e^{\pm \gamma_1}& =\pm e^{\pm \gamma_1},\quad  h(0)e^{\pm \gamma_2} =\pm e^{\pm \gamma_2},
\notag
\\
x^{\pm}(0)e^{\mp\gamma_1}&=\mp e^{\pm\gamma_2},\quad  x^{\pm}(0)e^{\mp\gamma_2}=\pm e^{\pm\gamma_1}.
\label{xzerogam}
\end{align}
Define $a,\abar,b,\bbar \in (V_{\ZZ^6})_{1}$ by
\begin{align*}
a&=\frac{1}{\sqrt{2}}(\gamma_3{+}i\gamma_4),\quad 
\abar=\frac{1}{\sqrt{2}}(\gamma_3{-}i\gamma_4),\\
b&=\frac{1}{\sqrt{2}}(\gamma_5{+}i\gamma_6),\quad 
\bbar=\frac{1}{\sqrt{2}}(\gamma_5{-}i\gamma_6),
\end{align*}
which satisfy non-zero relations $a(1)\abar{=}b(1)\bbar{=}\vac$.\ 
Lastly, define $\tau^{\pm},\overline{\tau}^{\pm}$ by 
\begin{align*}
\tau^{+}&=a(-1)e^{\gamma_1}+b(-1)e^{\gamma_2},\quad
\tau^{-}= a(-1)e^{-\gamma_2}-b(-1)e^{-\gamma_1},\\
\overline{\tau}^{+}&=\abar(-1)e^{\gamma_2} -\bbar(-1)e^{\gamma_1},\quad
\overline{\tau}^{-}=-\abar(-1)e^{-\gamma_1}-\bbar(-1)e^{-\gamma_2}.
\end{align*}

We now show that the sub-SVOA generated by $\tau^{\pm}$, $\overline{\tau}^{\pm}$ is isomorphic to  $ \Nfour$ for central charge $c{=}6$ by use of Theorem~\ref{thmUN=4}.\
$\tau^{\pm}$, $\overline{\tau}^{\pm}$ are clearly primary vectors of weight $\frac{3}{2} $. 
It is straightforward to confirm \eqref{xpmtau} and \eqref{xpmtau0} by using \eqref{xzerogam}. Thus Axiom~(I)  of Theorem~\ref{thmUN=4} holds.

\medskip
In order to confirm Axioms~(II)--(IV) of Theorem~\ref{thmUN=4}  we  note, using superassociativity, that for all 
$u,v{\in} \{a,b\}$ and $\lambda,\nu{\in}\{\pm \gamma_1,\pm\gamma_2\}$  
\begin{align}
&\left(u(-1)e^{\lambda}\right)(1)v(-1) e^{\nu}
=\dlangle  u,v\drangle  e^{\lambda}(-1)e^{\nu},
\label{vu1}
\\
&\left(u(-1)e^{\lambda}\right)(0)v(-1) e^{\nu}
=u(-1)v(-1)e^{\lambda}(0)e^{\nu}+\dlangle  u,v\drangle  e^{\lambda}(-2)e^{\nu},
\label{vu0}
\end{align}
where $u(1)v{=}\dlangle  u,v\drangle \vac$. 
\eqref{vu1} and Subsection~\ref{VL} imply  that
\begin{align*}
&\tau^{+}(1)\overline{\tau}^{-}=-\varepsilon(\gamma_1,-\gamma_1)\gamma_1(-1)\vac-\varepsilon(\gamma_2,-\gamma_2)\gamma_2(-1)\vac=h,
\\
&\tau^{-}(1)\overline{\tau}^{+}=\varepsilon(-\gamma_2,\gamma_2)(-\gamma_2)(-1)\vac+\varepsilon(-\gamma_1,\gamma_1)(-\gamma_1)(-1)\vac=h.
\end{align*}
Therefore $\tau^{+}(1)\overline{\tau}^{+}{=}{-}2x^{+}$ and $\tau^{-}(1)\overline{\tau}^{-}{=}2x^{-}$ using  
$x^{\pm}(0)h{=}\mp 2x^{\pm}$ and hence $A(1)B\cong \slt$ i.e. Axiom~(II) holds. 
Axiom~(III) follows  from 
\begin{align*}
\tau^{+}(0)\tau^{-}{=}{-}\varepsilon(-\gamma_1,-\gamma_1)a(-1)b(-1)\vac {+}
\varepsilon(\gamma_2,-\gamma_2)b(-1)a(-1)\vac{=}0.
\end{align*}
using \eqref{vu0}. By skew-symmetry $\tau^{-}(0)\tau^{+}{=}\tau^{+}(0)\tau^{-}{=}0$ and so $A(0)A{=}0$ using $\slt$ symmetry. A similar argument applies showing that $B(0)B{=}0$.\ Lastly 
\begin{align*}
&\tau^{+}(0)\overline{\tau}^{+}{=}{-}(\gamma_1{+}\gamma_2)(-1)e^{\gamma_1{+}\gamma_2}{=}{-}Tx^{+},
\end{align*}
so  that Axiom~(IV) holds. Hence the SVOA generated by $\tau^{\pm}$, $\overline{\tau}^{\pm}$ is isomorphic to $ \Nfour$ for central charge $c{=}6k=6$ by Theorem~\ref{thmUN=4}.

\medskip
 To finish, we show that $\sigma{=}\omega$, the standard $V_{L}$  Virasoro vector of central charge $6$. 
From \eqref{vu0} we find
\begin{align*}
\sigma&=\frac{1}{2}\left(\tau^+(0)\overline{\tau}^-{-}\tau^-(0)\overline{\tau}^+\right)
\\
&=-\frac{1}{2}\Bigg(\varepsilon(\gamma_1,-\gamma_1)a(-1)\abar {+} e^{\gamma_1}(-2)e^{-\gamma_1}
+\varepsilon(\gamma_2,-\gamma_2)b(-1)\bbar {+} e^{\gamma_2}(-2)e^{-\gamma_2}
\\
&\quad {+} \varepsilon(-\gamma_2,\gamma_2)a(-1)\abar {+} e^{-\gamma_2}(-2)e^{\gamma_2}
{+} \varepsilon(-\gamma_1,\gamma_1)b(-1)\bbar {+} e^{-\gamma_1}(-2)e^{\gamma_1}
\Bigg)
\\
&=a(-1)\abar{+}b(-1)\bbar{+}\frac{1}{2}\left(\gamma_1(-1)\gamma_1{+}\gamma_2(-1)\gamma_2\right) 
=\frac{1}{2}\sum_{i=1}^{6}\gamma_i(-1)\gamma_i=\omega.
\end{align*}
Thus we  conclude that 
\begin{prop}  $V_{\ZZ^6}$ contains an $N{=}4$ superconformal subalgebra  with the standard lattice Virasoro vector for central charge $c{=}6$.
\end{prop}
\medskip

\subsection{Example 2.}\ 
Let $\alpha_1, \hdots, \alpha_6$ be an orthogonal basis for $\RR^6$ consisting of
vectors of norm $3$, and
let $L$ be the lattice spanned by the $\alpha_i$ together with
\begin{align*}
h{:=}\frac{1}{3}(\alpha_1{+}\hdots{+}\alpha_6){\in} L_{2}. 
\end{align*}
Then $L$ is an odd, positive-definite, integral lattice with \emph{theta function} 
\begin{align*}
\theta_L(\tau){=}1{+}2q{+}24q^{\frac{3}{2} }{+}\hdots
\end{align*}
In particular, there are no vectors of norm $1$, and $\pm h$ are the only roots.
\begin{prop}\label{thmalt} $V_L$ contains an $N{=}4$ superconformal subalgebra $A$ such that the Virasoro vector of $A$
is the standard lattice Virasoro vector of central charge $c{=}6$.
\end{prop}
\begin{proof}
Let $X{\subseteq}L_3$ consist
of the 6 vectors $\alpha_i$.\ In the formalism of Section \ref{SRecog2}, especially (\ref{tauvec}), we take $Y$ to consist
of the vectors $\{h{-}\alpha_i\}$, so that
the four generating states of $A$ of weight $\frac{3}{2}$ will be chosen to take the form
\begin{align*}
&\tau^+{:=}\sum_{\alpha{\in}X}c_{\alpha}e^{\alpha},\quad  \tau^-{:=}\sum_{\alpha{\in}X} c_{\alpha}\varepsilon(h, \alpha)e^{\alpha-h},\\ 
&\overline{\tau}^+{:=}\sum_{\alpha\in X} d_{h{-}\alpha}e^{h-\alpha},\quad
 \overline{\tau}^-{:=}-\sum_{\alpha\in X} d_{h-\alpha}\varepsilon(h, \alpha)e^{-\alpha}.
\end{align*}

We show that the hypotheses of Theorem~\ref{thmVLV} hold for certain choices of scalars 
$c_{\alpha}, d_{h-\alpha}, (\alpha{\in}X)$.\ Conditions (a) and (b) of Lemma~\ref{lemA(1)A} and condition (b)  of  Lemma~\ref{lemP} automatically hold since $X\cap Y=0$ and $\pm h$ are the only roots in $L$.\ We may check by  direct calculation that these facts \emph{imply} the assumptions of these two Lemmas (cf.\ the proofs of the Lemmas), i.e., 
$A(1)A=B(1)B=0$.
Now choose the scalars $c_{\alpha}, d_{h-\alpha}$ so that 
\begin{align*}
c_{\alpha}d_{h-\alpha}\varepsilon(h,  \alpha)=-\frac{1}{3}
\end{align*}
for each $\alpha_i{\in} X$,  implying condition (a) of Lemma~\ref{lemP}.\ Hence $A(1)B\cong \slt$.

\medskip
Because ${\pm h}$ are the only roots of $L$, there is a unique simple component of the Lie algebra
$(V_L)_1$, and it is isomorphic to $sl_2$.\ It follows that this component is our Lie algebra
$A(1)B{=}\slt$, and in particular $h{\in}\slt$ and
$A(1)B{=}\langle h, e^{\pm h}\rangle$.\ It is then straightforward to check that $A$ and $B$ are indeed
vector representations of $\slt$, so that all hypotheses of Theorem \ref{thmVLV} are satisfied.\
This completes the proof that  $\tau^{\pm},\ \overline{\tau}^{\pm}$ generate an $N{=}4$ subalgebra 
of $V_L$ with $c{=}6$.

\medskip
The Virasoro vector in this example is (cf.\ Lemmas~\ref{lemsigma}, \ref{lemvir})
\begin{align*}
\sigma&{=}\frac{1}{2}(\tau^+(0)\overline{\tau}^-{-}\tau^-(0)\overline{\tau}^+)
\\
&{=}-\frac{1}{2}\sum_{\alpha, \beta{\in}X}\Big\{c_{\alpha} d_{h-\beta} \varepsilon(h, \beta)e^{\alpha}(0)e^{-\beta}+ c_{\alpha} d_{h{-}\beta} \varepsilon(h, \alpha)e^{\alpha-h}(0)  e^{h-\beta}\Big\}
\\
&{=}-\frac{1}{2} \sum_{\alpha{\in}X}\Big\{c_{\alpha} d_{h-\alpha} \varepsilon(h, \alpha)e^{\alpha}(0)e^{-\alpha} 
+c_{\alpha} d_{h{-}\alpha} \varepsilon(h, \alpha)e^{\alpha-h}(0)  e^{h-\alpha}\Big\}
\\
&{=}\frac{1}{12} \sum_{\alpha{\in}X}\Big(\alpha(-2) +\alpha(-1)^2 +(\alpha-h)(-2) +(\alpha-h)(-1)^2 \Big)\vac
\\
&{=}\frac{1}{12} \sum_{\alpha{\in}X}\Big(2\alpha(-1)^2 +h(-1)^2 -2\alpha(-1)h(-1)+2\alpha(-2)-h(-2) \Big)\vac
\\
&{=}\frac{1}{6}\sum_{\alpha} \alpha(-1)\alpha.
 \end{align*}
This is indeed the standard Virasoro element  of $V_L$, and the proof is complete.
\end{proof}

\medskip

We next consider three examples of $N{=}2$ superconformal subalgebras of odd lattice SVOAS, illustrating Theorem~\ref{thmUN=2}.

\medskip
\subsection{Example 3.} Consider the   lattice SVOA $V_{\ZZ^3}$, the rank $6$ free fermion construction. Let $L={\ZZ^3}$ be generated by $\gamma_1,\gamma_2,\gamma_3\in L_{1}$ with  $(\gamma_i,\gamma_j)=\delta_{ij}$.  
Define $h,a^{\pm} \in (V_{\ZZ^3})_{1}$ by
\begin{align*}
h&=\gamma_1,\qquad 
a^{\pm}{=}\frac{1}{\sqrt{2}}(\gamma_2\pm i\gamma_3),
\end{align*}
with $a^{+}(1)a^{-}{=}\vac$. 
Lastly, define $\tau^{\pm}{=}a^{\pm}(-1)e^{\pm \gamma_1}$.

Using \eqref{vu0} and \eqref{vu1} we find that Axioms~(I) and (II) of Theorem~\ref{thmUN=2} hold.\
Since $h(1)h{=}\vac$, the central charge is $c{=}3$.\ Using \eqref{eq:N2Vir} one finds that 
\begin{align*}
\sigma{=}a^{+}(-1)a^{-} {+}\frac{1}{2}\gamma_1(-1)\gamma_1=\frac{1}{2}\sum_{i=1}^{3}\gamma_i(-1)\gamma_i,
\end{align*}
the standard lattice Virasoro vector for $V_{\ZZ^3}$.\ This establishes
\begin{prop}  $V_{\ZZ^3}$ contains an $N{=}2$ superconformal subalgebra  with the standard lattice Virasoro vector for central charge $c{=}3$.
\end{prop}

\medskip

\noindent
\subsection{Example~4.} In this example we show that every  odd lattice SVOA for which $L_3\neq \emptyset$ contains an $N{=}2$ SVOA with central charge $c{=}1$.
 \begin{prop}\label{thmmoreN=2} Let $\gamma{\in}L_3$ and define $\tau^{\pm}{:=}\frac{1}{\sqrt{3}}e^{\pm\gamma }$ and $h{:=}\frac{1}{3}\gamma$.\ 
Then $\tau^{\pm}$ generate an $N{=}2$ subalgebra of $V_L$ with Virasoro vector 
$\omega{=}\frac{1}{6}\gamma(-1)\gamma$ and $c{=}1$.
\end{prop}
\begin{proof} 
We  find $h{=}\tau^{\pm}(1)\tau^{\mp}$  with $h(0)\tau^{\pm}{=}\pm \tau^{\pm}$ and $\tau^{\pm}(0)\tau^{\pm}{=}0$. Thus  Axioms~(I) and (II) of  Theorem~\ref{thmUN=2} hold, so $\tau^{\pm}$ generate an $N{=}2$ superconformal algebra.\  Since $h(1)h{=} \frac{1}{3}\vac$,  the central charge is $c{=}1$ with  Virasoro vector $\frac{1}{6}\gamma(-1)\gamma$ from \eqref{eq:N2Vir}. 
\end{proof}

\noindent
\subsection{Example~5.}
\begin{prop}\label{thmN=2B} Let $\alpha, \beta{\in}L_3$ with $(\alpha, \beta){=}1$, 
and let $\lambda, \mu$ be nonzero scalars.\  Define
\begin{align*}
\tau^{+}&{=}\frac{1}{2}(\lambda e^{\alpha}{+}\mu e^{\beta}),\quad 
 \tau^-{=}\frac{1}{2}(\lambda^{-1}e^{-\alpha}{+}\mu^{-1} e^{-\beta}),\quad
h{=}\frac{1}{4}(\alpha{+}\beta),\\
\omega&{=}\frac{1}{8}\left\{\alpha(-1)\alpha{+}\beta(-1)\beta{+}
\frac{2\lambda}{\mu} \varepsilon(\alpha, \beta)e^{\alpha-\beta}{+}\frac{2\mu}{\lambda} 
\varepsilon(\beta, \alpha)e^{\beta-\alpha}\right\}.
\end{align*}
Then
$\tau^{\pm}$ generate an $N{=}2$ subalgebra of $V_L$ with $c{=}\frac{3}{2}$.
\end{prop}
\begin{proof}
Axioms~(I) and (II) of Theorem~\ref{thmUN=2} are easily seen to hold. $h(1)h= \frac{1}{2}\vac$ implies  the central charge $c{=}\frac{3}{2}$ and  $\omega$ is as given on applying \eqref{eq:N2Vir}.
\end{proof}

\medskip\noindent

\section{Appendices}\label{SAppendix}

\subsection{Axioms for super VOAs} 
The underlying Fock space is a $\half\ZZ$--graded
 $\CC$-linear super vector space 
\begin{align*}
V{=} \oplus_{k\in\half \ZZ}V_{k},
\end{align*}
with \emph{parity} operator $p(u)=  2k\bmod 2$ for $u\in V_{k}$.
Each state $u{\in}V$ has a vertex operator $Y(u, z){=}\sum_{n{\in}\ZZ} u(n)z^{-n-1}$;\ $u(n){\in}\End(V)$ is the $n^{th}$ \emph{mode} of $u$.

\medskip
There is a distinguished \emph{vacuum state} $\mathbf{1}{\in}V_0$ with vertex operator
$Y(\mathbf{1}, z){=}\Id_V$; and a distinguished \emph{Virasoro state}
$\omega{\in}V_2$ with vertex operator
\begin{align*}
Y(\omega, z){=}\sum_{n{\in}\ZZ} L(n)z^{-n-2},
\end{align*}
whose modes satisfy the Virasoro relations with \emph{central charge} $c${:}
\begin{align*}
[L(m), L(n)]{=}(m{-}n)L(m{+}n)+\frac{1}{2}\binom{m{+}1}{3}\delta_{m+n, 0}c \Id_V.
\end{align*}
We distinguish the endomorphism $T{\in}\End(V)$ defined by $T(u){:=}u(-2)\mathbf{1}{=}L(-1)u$. 
The $\half \ZZ$ grading is determined by $L(0)$ with $L(0)u=ku$ for $u\in V_k$.

\medskip
Modes satisfy the following axioms, the third being the super Jacobi identity{:}
\begin{align*}
&(a)\ u(n)v=0\ \mbox{for all}\ n{\geq}n_0,\\
&(b)\ u(-1)\mathbf{1}{=}u;\ u(n)\mathbf{1}{=}0\ \mbox{for}\ n{\geq}0,\\
&(c)\ \forall\, r, s, t{\in}\ZZ, \\
&\ \ \ \ \ \ \ \ \ \ \ \ \ \ \ \ \ \ \ \  \sum_{i\geq 0} \binom{r}{i}(u(t{+}i)v)(r{+}s{-}i) w= \\
&\ \ \ \ \ \sum_{i\geq 0}(-1)^i \binom{t}{i} \left\{u(r{+}t{-}i)v(s{+}i){-}(-1)^{t{+}p(u)p(v)}v(s{+}t{-}i)u(r{+}i)\right\}w,
\end{align*}
for all $u,v,w\in V$.
The special cases $t{=}0, r{=}0$ give respectively super commutativity
\begin{align*}
u(r)v(s){-}(-1)^{p(u)p(v)}v(s)u(r){=} \sum_{i\geq 0} \binom{r}{i}(u(i)v)(r{+}s{-}i),
\end{align*}
and super associativity
\begin{align*}
&\ \ \ \  (u(t)v)(s) {=} \sum_{i\geq 0}(-1)^i \binom{t}{i} \left\{u(t{-}i)v(s{+}i){-}({-}1)^{t{+}p(u)p(v)}v(s{+}t{-}i)u(i)\right\}.
\end{align*}
Taking $r{=}{-}1, s{=}0, w{=}\mathbf{1}$ leads to super skew-symmetry
\begin{align*}
v(t)u{=} {-}(-1)^{t{+}p(u)p(v)}\sum_{i\geq 0}(-1)^iT^iu(t{+}i)v.
\end{align*}

\subsection{The $\varepsilon$-formalism}\label{SSepsilon}
Fix a finitely generated free abelian group $L$.\ We are interested
in groups $\widehat{L}$
which are \emph{central extensions of $L$ by $\ZZ_2$}.\ 
So there is a short exact sequence of groups
\begin{align}\label{ses}
1\rightarrow \{\pm 1\}\rightarrow \widehat{L} \stackrel{\overline{\;\;}}{\rightarrow} L\rightarrow 0,
\end{align}
and $\widehat{L}$ can be identified with $L{\times} \{\pm 1\}$ as a set, with multiplication
\begin{align*}
(\alpha, e)(\beta, f) = (\alpha{+}\beta, \varepsilon(\alpha, \beta)ef)\qquad(\alpha, \beta{\in} L, e, f{\in} \{{\pm} 1\}),
\end{align*}
where
\begin{eqnarray*}
\varepsilon{:}L{\times}L\rightarrow \{\pm 1\}.
\end{eqnarray*}
We may, and shall, take $\varepsilon$
to be \emph{bimultiplicative}, i.e.,
\begin{align*}
\varepsilon(\alpha{+}\beta, \gamma)=\varepsilon(\alpha, \gamma)\varepsilon(\beta, \gamma),
\quad \varepsilon(\alpha, \beta{+}\gamma)=\varepsilon(\alpha, \beta)\varepsilon(\alpha, \gamma).
\end{align*}
This ensures that $\varepsilon{\in}Z^2(L, \{\pm 1\})$ is a \emph{2-cocycle} and that
 multiplication in $L$ is \emph{associative}.
In particular we note that $\varepsilon(\alpha, 0){=}\varepsilon(0, \alpha){=}1$ and  
$\varepsilon(\alpha, \beta){=}\varepsilon(\alpha, -\beta){=}\varepsilon(-\alpha, \beta)$.

\medskip
If $L$ is a positive-definite integral lattice with bilinear form $(\ , \ )$, then we may further choose $\varepsilon$ (\cite{K} P.\ 155) so that it satisfies
\begin{align*}
\varepsilon(\alpha,\beta)\varepsilon(\beta, \alpha)=(-1)^{(\alpha, \beta)+(\alpha,\alpha)(\beta, \beta)},
\ \ \ \varepsilon(\alpha, \alpha) = (-1)^{((\alpha, \alpha)+(\alpha, \alpha)^2)/2}.
\end{align*}

\begin{rmk} Depending on context, various choices for $\varepsilon$ are used in the literature, although they give equivalent theories.\
The one used in
\cite{FLM}, for example, is different to the one we generally use here.
\end{rmk}

\subsection{Super lattice theories}\label{VL}
Let $L$ be a positive-definite integral lattice equipped with a bilinear form $(\ ,\ )$, with
$\varepsilon$ as in Subsection \ref{SSepsilon}.\ The
\emph{twisted group algebra} $\CC^{\varepsilon}[L]$ has basis $e^{\alpha} (\alpha{\in}L)$ and multiplication
\begin{align*}
&e^{\alpha}e^{\beta}=\varepsilon(\alpha,\beta)e^{\alpha+\beta}\ \ \ \ (\alpha,\beta\in L).
\end{align*}
It is $\half \ZZ$-graded by $wt(e^{\alpha}){:=}\half(\alpha, \alpha)$.

\medskip
There are Lie algebras (the first is abelian)
\begin{align*}
&\mathfrak{h}{:=}\CC\otimes_{\mathbb{Z}} L,\ \  \widehat{\mathfrak{h}}{:=}\mathfrak{h}\otimes \CC[t,t^{-1}]\oplus\CC c,\ \ \mathfrak{h}^+=\mathfrak{h}\otimes t\CC[t], \ \  \widehat{\mathfrak{h}}^-=\mathfrak{h}\otimes t^{-1}\CC[t^{-1}],
 \end{align*}
with brackets 
$[x\otimes t^m,y\otimes t^n]{=}(x,y) m\delta_{m+n,0}c,~[c,\widehat{\mathfrak{h}}]{=}0$,\
and an induced $\widehat{\mathfrak{h}}$-module
$$M(1)=U(\widehat{\mathfrak{h}})\otimes_{U(\mathfrak{h}\otimes \CC [t]\oplus \CC c)}\CC \cong S(\widehat{\mathfrak{h}}^-) \text{ (linearly)},$$
$\mathfrak{h}\otimes \CC[t]$ acting trivially on $\CC$ and $c$ acting as 1.\  Fock space for the lattice theory is
$$V_L=M(1){\otimes}\CC^{\varepsilon}[L]\cong S(\widehat{\mathfrak{h}}^-){\otimes} \CC [L]~\text{(linearly)}$$
with the usual tensor product grading.\ The Virasoro vector is $\half \sum_i h_i(-1)h_i$, the sum ranging
over \emph{any} orthonormal basis $\{h_i\}$ of $\mathfrak{h}$.

\medskip
 For $\alpha{\in}\mathfrak{h}$ write $\alpha(n){:=}\alpha\otimes t^n,\ \alpha(z){:=}\sum_{n\in\mathbb{Z}}\alpha(n)z^{-n-1}, z^{\alpha}{:}e^{\beta}{\mapsto} z^{(\alpha, \beta)}e^{\beta}$, and set
$$Y(e^{\alpha},z){:=} \exp\left(\sum_{m=1}^{\infty}\alpha(-m)\frac{z^m}{m}\right)\exp\left(-\sum_{m=1}^{\infty}\alpha(m)\frac{z^{-m}}{m}\right)e^{\alpha}z^{\alpha},$$
and for
$v=\alpha_1(-n_1)...\alpha_k(-n_k){\otimes}e^{\alpha}\in V_L\ (n_i{\geq}1)$ set
\begin{align*}
&Y(v,z){:=}{:}\left(\frac{1}{(n_1-1)!}\left(\frac{d}{dz}\right)^{n_1-1}\alpha_1(z)\right)...\left(\frac{1}{(n_k-1)!}\left(\frac{d}{dz}\right)^{n_k-1}\alpha_k(z)\right)Y(e^{\alpha},z){:},
\end{align*}
with the usual normal ordering conventions.


\medskip
For $\gamma, \rho{\in}L$ we have
\begin{align*}
e^{\gamma}(n)e^{\rho}= \left \{ \begin{array}{cccc}
0 & \mbox{if}\ (\gamma, \rho)\geq -n \\
\varepsilon(\gamma, \rho)e^{\gamma+\rho}   &\ \ \ \ \ \  \mbox{if}\ (\gamma, \rho)=-n-1 \\
\varepsilon(\gamma, \rho)\gamma(-1)e^{\gamma+\rho}   &\ \ \ \ \ \  \mbox{if}\ (\gamma, \rho)=-n-2 \\
\frac{1}{2}\varepsilon(\gamma, \rho)(\gamma(-2)e^{\gamma+\rho}+\gamma(-1)^2e^{\gamma+\rho})   &\ \ \ \ \ \  \mbox{if}\ (\gamma, \rho)=-n-3
  \end{array} \right.
\end{align*}

\end{document}